\documentclass[11pt,letterpaper]{article}%
\usepackage{amsmath}
\usepackage{amsfonts}
\usepackage{amssymb}
\usepackage{graphicx}
\usepackage{pstricks}
\usepackage{pst-node}
\usepackage{pst-plot}
\usepackage[letterpaper]{geometry}
\newtheorem{theorem}{Theorem}
\newtheorem{proposition}[theorem]{Proposition}
\newtheorem{corollary}[theorem]{Corollary}
\newtheorem{lemma}[theorem]{Lemma}

\newenvironment{proof}[1][Proof]{\noindent \textbf {#1.} }{\ \rule {0.5em}{0.5em}}
\newtheorem{preremark}[theorem]{Remark}
\newenvironment{remark}{\begin {preremark}\rm }{\hfill $\Diamond $\end {preremark}}
\newtheorem{prenotation}[theorem]{Notation}

\newenvironment{t_enumerate}{ \begin {enumerate} \setlength {\itemsep }{1pt} \setlength {\parskip }{0pt} \setlength {\parsep }{0pt}}{\end {enumerate} }

\numberwithin{equation}{section}
\numberwithin{theorem}{section}

\begin{document}

\title{{Coherent state transforms and the Mackey-Stone-Von Neumann theorem}}
\author{William D. Kirwin\thanks{Mathematics Institute, University of Cologne, Weyertal 86 - 90, 50931 Cologne, Germany.\newline email: will.kirwin@gmail.com}, Jos\'e M. Mour\~ao and Jo\~ao P. Nunes\thanks{Center for Mathematical Analysis, Geometry and Dynamical Systems and Department of Mathematics, Instituto Superior T\'ecnico, Av. Rovisco Pais, 1049-001 Lisbon, Portugal.\newline email: jmourao@math.ist.utl.pt, jpnunes@math.ist.utl.pt}}
\date{}
\maketitle

\begin{abstract}
Mackey showed that for a compact Lie group $K$, the pair $(K,C^{0}(K))$ has a unique non-trivial irreducible covariant pair of representations. We study the relevance of this result to the unitary equivalence of quantizations for an infinite-dimensional family of $K\times K$ invariant polarizations on $T^{\ast}K$. The K\"{a}hler polarizations in the family are generated by (complex) time-$\tau$ Hamiltonian flows applied to the (Schr\"{o}dinger) vertical real polarization. The unitary equivalence of the corresponding quantizations of $T^{\ast}K$ is then studied by considering covariant pairs of representations of $K$ defined by geometric prequantization and of representations of $C^0(K)$ defined via Heisenberg time-$(-\tau)$ evolution followed by time-$(+\tau)$ geometric-quantization-induced evolution. We show that in the semiclassical and large imaginary time limits, the unitary transform whose existence is guaranteed by Mackey's theorem can be approximated by composition of the time-$(+\tau)$ geometric-quantization-induced evolution with the time-$(-\tau)$ evolution associated with the momentum space \cite{Kirwin-Wu12} quantization of the Hamiltonian function generating the flow. In the case of quadratic Hamiltonians, this asymptotic result is exact and unitary equivalence between quantizations is achieved by identifying the Heisenberg imaginary time evolution with heat operator evolution, in accordance with the coherent state transform of Hall.
\end{abstract}
\tableofcontents

\section{Introduction}
\label{intro}

Geometric quantization has proven to be a very rich approach to the general mathematical problem of the quantization of a symplectic manifold $(M,\omega)$. In order to half-form quantize $(M,\omega),$ one needs to choose a polarization ${{\mathcal{P}}}$, that is an involutive Lagrangian distribution in the complexified tangent bundle $TM\otimes{{\mathbb{C}}}$. The half-form quantization $\mathcal{H}_{\mathcal{P}}$ of $(M,\omega,\mathcal{P})$ is then the $L^{2}$-closure of the space of square-integrable smooth sections of the \emph{quantum bundle} $L\otimes\sqrt{\mathcal{K}}$ which are covariantly constant along $\overline{\mathcal{P}}$, where $\sqrt{\mathcal{K}}$ is a choice of square root of the canonical bundle $\bigwedge^{n}\mathcal{P}^{\ast}$ of $\mathcal{P}$ and $L$ is a Hermitian line bundle with compatible connection with curvature $-i\omega$. (Of course, in general there are topological obstructions to the existence of $\sqrt{\mathcal{K}}$ and $L$, but they will play no role in this paper.) A major, perhaps even the fundamental, issue in geometric quantization is the dependence of quantization on the choice of ${{\mathcal{P}}}$. In favorable cases, given two polarizations ${{\mathcal{P}}},{{\mathcal{P}}}^{\prime}$ of $(M,\omega)$, one would like to have a natural unitary isomorphism between the corresponding Hilbert spaces of quantum states ${{\mathcal{H}}}_{{\mathcal{P}}},{{\mathcal{H}}}_{{{\mathcal{P}}}^{\prime}}$, at least up to a projective ambiguity. Moreover, this unitary isomorphism should intertwine actions of sufficiently big algebras of observables on ${{\mathcal{H}}}_{{\mathcal{P}}}$ and ${{\mathcal{H}}}_{{{\mathcal{P}}}^{\prime}}$.

In the paradigmatic case when $(M,\omega)$ is a symplectic vector space and when one considers translation invariant polarizations, such an isomorphism is guaranteed to exist by the Stone--Von Neumann theorem, which gives uniqueness of the irreducible unitary representation of the Heisenberg group \cite{Axelrod-DellaPietra-Witten,Woodhouse,Kirwin-Wu06}. The linear observables, which generate translations and which, together with the constants, span the Heisenberg algebra, preserve the invariant polarizations and therefore have geometric-quantization induced actions on the Hilbert spaces $\mathcal{H}_{\mathcal{P}}$ of $\mathcal{P}$-polarized quantum states, which integrate to irreducible representations of the Heisenberg group. According to the Stone--Von Neumann uniqueness theorem, all such representations are unitarily equivalent and then, from Schur's lemma, there is a unique-up-to-phase unitary operator
\[
U_{\mathcal{P},\mathcal{P}^{\prime}}\ :\ \mathcal{H}_{\mathcal{P}}\longrightarrow\mathcal{H}_{\mathcal{P}^{\prime}},
\]
establishing equivalence of the quantizations for different polarizations within this class.

When $(M,\omega)$ is a symplectic torus $M=T^{2n}$, the corresponding Stone-Von Neumann type theorem for the finite Heisenberg group also guarantees the equivalence between quantizations for translation invariant polarizations \cite{Axelrod-DellaPietra-Witten,Polishchuk,BMN}.

In this paper, we will address this question in the case when $M=T^{\ast}K\overset{\pi}{\longrightarrow}K$ is the cotangent bundle of a compact Lie group $K$ equipped with the canonical symplectic form $\omega$. One motivation for this is to address a question raised by Hall in \cite{hall00,Hall01} on finding an analogue, for the quantization of $T^{\ast}K$, of the role played by Stone--Von Neumann theorem in establishing the unitary equivalence of quantizations for translation-invariant polarizations on a symplectic vector space. In the vector space case, the unitary equivalence between the real polarized and K\"{a}hler polarized quantum Hilbert spaces is realized by the Segal--Bargmann transform \cite{Kirwin-Wu06}. Indeed, Bargmann originally found the explicit form of the Segal--Bargmann transform by using the Stone--Von Neumann-guaranteed equivariance \cite{Bargmann}.\ In \cite{Hall94}, Hall describes a generalized Segal--Bargmann transform, which he calls the Coherent State Transform (CST), which is a unitary $K$-equivariant isomorphism between square-integrable functions on $K$ with respect to the Haar measure $dx$ and a certain weighted space of holomorphic functions on the complexified group $K_{\mathbb{C}}.$ In \cite{Hall02}, Hall showed that in fact $L^{2}(K,dx)$ and the weighted space of holomorphic functions which appear in his CST are the, respectively, real-polarized and K\"{a}hler quantizations of $T^{\ast}K$. As mentioned in \cite{hall00,Hall01}, it would be desirable to have a a better understanding of the existence of the CST in terms of a Stone--Von Neumann type result. We address this question in the present paper.

\bigskip

More specifically, we will study the role of Mackey's generalization of the Stone-Von Neumann theorem \cite{Mackey} in establishing the unitary equivalence of quantization for an infinite-dimensional family of $K\times K$-invariant K\"{a}hler polarizations on $T^{\ast}K$.

We consider the infinite-dimensional family ${{\mathcal{T}}}$ of $K\times K$-invariant K\"{a}hler polarizations of the cotangent bundle $T^{\ast}K$ of a compact Lie group $K$ which were studied in \cite{kirwin-mourao-nunes12}. The boundary of this family includes real polarizations and, in particular, the vertical polarization, which we denote by $\mathcal{P}_{Sch}$ as the corresponding quantum Hilbert space is the usual Schr\"{o}dinger quantization of $T^{\ast}K$. We show that even though the case of $T^{\ast}K$ is not quite the same as the case of a symplectic vector space, they share many similarities.

Let $\mathcal{H}_{Sch}=L^{2}(K,dx)\otimes\sqrt{dx}$ be the Hilbert space of half-form corrected polarized quantum states for $\mathcal{P}_{Sch}$, where $dx$ is the normalized Haar measure on $K$. Since $\mathcal{P}_{Sch}$ is preserved by the Hamiltonian vector fields corresponding to smooth functions on $K$ (pulled-back to $T^{\ast}K$) as well as by the generators of the $K\times K$ Hamiltonian action, geometric quantization, in particular the Kostant--Souriau formula with half-forms, defines on $\mathcal{H}_{Sch}$ a pair of representations $\rho^{Sch}:K\times K\rightarrow\mathcal{O}(\mathcal{H}_{Sch})$ and $\gamma^{Sch}:C^{0}(K)\rightarrow\mathcal{O}(\mathcal{H}_{Sch})$. The pair $(\rho_{|_{K\times\{e\}}}^{Sch},\gamma^{Sch})$ is naturally isomorphic to the standard covariant pair of representations of $K$ and $C^{0}(K)$ on $L^{2}(K,dx)$ \cite{Mackey,Rosenberg04}. (See Section \ref{s31}.) Indeed, Mackey's uniqueness theorem \cite{Mackey,Rosenberg04} states that there is only one such irreducible covariant pair up to equivalence.

Let $\operatorname{Conv}(K)$ be the set of smooth strictly convex $K\times K$-invariant functions on $T^{\ast}K$, such that the operator norm of the Hessian of any $h\in \operatorname{Conv}(K)$ has a positive lower bound. The family of polarizations we consider was studied in \cite{kirwin-mourao-nunes12} and consists of polarizations ${{\mathcal{P}}}_{\tau,h}\in{{\mathcal{T}}}$ \ labeled by a pair $(\tau,h),$ where
\[\tau\in{{\mathbb{C}}}^{+}=\{\tau\in{{\mathbb{C}}}:\tau=\tau_{1}+i\tau_{2},\tau_{2}>0\},
\]
and $h\in \operatorname{Conv}(K)$. The polarizations $\mathcal{P}_{\tau,h}$ are $K\times K$-invariant and are K\"ahler with respect to $J_{\tau,h}$, where $J_{\tau,h}$ is the corresponding complex structure on $T^*K$.

The family ${{\mathcal{T}}}$ can be partially compactified to
\begin{equation}\label{compact}
\overline
{{\mathcal{T}}}={{\mathcal{T}}}\cup\{{{\mathcal{P}}}_{t,h},t\in{{\mathbb{R}},h\in \operatorname{Conv}(K)}\},
\end{equation}
by including an infinite-dimensional family of real polarizations obtained from the vertical polarization by pushing forward by the real-time Hamiltonian flow of $h\in \operatorname{Conv}(K)$.\footnote{Note that the conditions we impose on the Hamiltonian functions $h$ are only relevant when considering complex-time Hamiltonian flow, that is for $\tau_{2}>0$.} For any Hamiltonian function $h$ satisfying the conditions above, $\mathcal{P}_{0,h}=\mathcal{P}_{Sch}$ is the vertical polarization, and this is the only polarization in the family which is also invariant under $C^{\infty}(K)$. The other real polarizations in the family are given by the points $\mathcal{P}_{t,h},t\in{{\mathbb{R}}}$, and are invariant under algebras of functions on $T^{\ast}K$ which are $\ast$-isomorphic to $C^{\infty}(K)$.

Even though $X_{h}$ does not preserve $\mathcal{P}_{Sch}$, Hamiltonian vector fields associated to $K$-invariant functions which are linear in the fiber variables$\ $\emph{do} preserve $\mathcal{P}_{Sch}$ and therefore act on ${{\mathcal{H}}}_{Sch}$; in particular, they generate the right $K$ action
\begin{equation}
\hat{y}^{j}\cdot\left(  \psi(x)\otimes\sqrt{dx}\right)  =i(X_{j}\cdot\psi(x))\otimes\sqrt{dx},\label{schaction}
\end{equation}
where $\{y^{j}\}_{j=1,\dots,n}$ are coordinates for an orthonormal basis of left-invariant vector fields on $K$, with $\{{\hat{y}}^{j}\}_{j=1,\dots,n}$ the corresponding Kostant--Souriau prequantum operators. Similarly, taking the Kostant--Souriau operators associated to an orthonormal basis of right-invariant vector fields on $K$, $\{{\hat{\tilde y}}^{j}\}_{j=1,\dots,n}$, one obtains the left $K$ representation on ${{\mathcal{H}}}_{Sch}$ \cite{kirwin-mourao-nunes12}. The decomposition of $\mathcal{H}_{Sch}$ into irreducible representations with respect to this $K\times K$ action is the usual Peter--Weyl decomposition
\[\mathcal{H}_{Sch}\simeq L^{2}(K,dx)\simeq\bigoplus_{\lambda\in\hat{K}}V_{\lambda\otimes\lambda^{\ast}},\]
where $\hat{K}$ is the set of equivalence classes of irreducible unitary representations of $K$.

The $K\times K$-invariance of $\mathcal{P}_{\tau,h}$ implies that for each $(\tau,h)$, the Hilbert space $\mathcal{H}_{\tau,h}$ of quantum states for $\mathcal{P}_{\tau,h}$ carries a $K\times K$ representation, and it was shown in \cite{kirwin-mourao-nunes12} that $\mathcal{H}_{\tau,h}$ also decomposes into unitary irreducible representations of $K\times K$ as
\[
\mathcal{H}_{\tau,h}\simeq \bigoplus_{\lambda\in\hat{K}}V_{\lambda\otimes \lambda^{\ast}}.
\]
Hence, there exist unitary maps intertwining the $K\times K$ actions on $\mathcal{H}_{Sch}\simeq L^{2}(K,dx)\simeq\bigoplus_{\lambda\in \hat{K}}V_{\lambda\otimes\lambda^{\ast}}$ and $\mathcal{H}_{\tau,h}$, and these unitary maps are determined up to a choice of phase for each $\lambda\in\hat{K}$. If geometric quantization also defined a $*$-representation of $C^0(K)$ on $\mathcal{H}_{\tau,h}$ thus completing the previous $K\times \{e\}$ representation to an irreducible covariant pair, then, via Mackey's theorem, one would be in the same situation as in the quantization of a symplectic vector space with invariant polarizations. However, this is not the case since, as mentioned above, observables in $C^\infty(K)$ do not preserve $\mathcal{H}_{\tau,h}$. Nevertheless, albeit indirectly, geometric quantization still allows us to define a representation of $C^0(K)$ on $\mathcal{H}_{\tau,h}$. As we will show, in the case when $h$ is quadratic the projective unitarity of the CST of Hall is equivalent to the $*$-property of this representation, so that  the CST can be understood in the context of Mackey's theorem.

The construction of representations of $C^0(K)$ on ${\mathcal H}_{\tau,h}$ can be naturally divided into three parts.

\smallskip
\noindent\textbf{1. Representation of ${\mathcal A}_{\tau,h}$ on $\mathcal{H}_{\tau,h}$.} (See Theorem \ref{thlifts}.) The Kostant--Souriau prequantization $\hat{h}$ of the Hamiltonian function $h$ gives a densely defined operator \cite{kirwin-mourao-nunes12},
\begin{equation}
e^{-i\tau\hat{h}}:\mathcal{H}_{Sch}\rightarrow\mathcal{H}_{\tau,h},\label{hhat}
\end{equation}
that intertwines the actions of $K\times K$. Let ${\mathcal A}_0$ be the algebra of functions on $K$ generated by matrix elements of irreducible representations of $K$ pulled back to $T^*K$ by the canonical projection. The action of ${\mathcal A}_0$ on $\mathcal{H}_{Sch}$ is intertwined by $e^{-i\tau \hat{h}}$ with the action of an algebra of $J_{\tau,h}$-holomorphic functions on $T^{\ast}K$,
\begin{equation}
{\mathcal A}_{\tau,h}=e^{\tau X_{h}}{\mathcal A}_0=\{e^{\tau X_{h}}\cdot f:f\in{\mathcal A}_0\},\label{cxevol}
\end{equation}
where $X_{h}$ is the Hamiltonian vector field for $h$. Here, $e^{\tau X_{h}}\cdot f$ denotes the analytic continuation of $f_{|_K}$ from $K$ to $(T^*K,J_{\tau,h})$. Note that this expression can be literally interpreted as a power series in $\tau$ \cite{Hall-Kirwin, kirwin-mourao-nunes12}.

\smallskip
\noindent\textbf{2. Representation of ${\mathcal A}_{-\tau,h}$ on $\mathcal{H}_{Sch}$.} (See Theorem \ref{39} and equation (\ref{gaztil}).) As we have just seen, the operator $e^{-i\tau\hat{h}}$ evolves states from $\mathcal{H}_{Sch}$ to $\mathcal{H}_{\tau,h}$ but also evolves observables as in (\ref{cxevol}). It is then natural to expect that in order to have $e^{-i\tau\hat{h}}$ defining a representation of ${\mathcal A}_0$ on  $\mathcal{H}_{\tau, h}$ we should start from a representation of ${\mathcal A}_{-\tau,h} \ = \ e^{-\tau X_{h}}{\mathcal A}_0$ on $\mathcal{H}_{Sch}$. We will achieve this by choosing a representation $Q(h)$ of $h$ on $\mathcal{H}_{Sch}$ and then by taking Heisenberg evolution in complex time $-\tau$.
\begin{equation}
 {\mathcal A}_{-\tau,h} \ni e^{-\tau X_{h}}\cdot f\mapsto e^{i\tau Q(h)}\circ f\circ e^{-i\tau Q(h)},\label{heisenberg}
\end{equation}
for $f\in {\mathcal A}_0$, acting as operators on ${{\mathcal{H}}}_{Sch}=L^{2}(K,dx)$.

Let us now motivate our choice of $Q(h)$. We will identify $Lie(K)\cong Lie(K)^*$ via the invariant bilinear form on $Lie(K)$ corresponding to the normalized Haar measure on $K$. In \cite{Kirwin-Wu12}, the first author and Wu have shown that there is a momentum space polarization for $T^*K$ which corresponds to $\lim_{\tau_2\to \infty} {\mathcal{P}}_{{i\tau_2},\frac{|Y|^2}{2}}.$ Moreover, they have shown that the half-form quantization corresponding to this polarization is given by Bohr--Sommerfeld fibers which are localized along submanifolds $K\times {\mathcal O}_{-(\lambda+\rho)}=K\times {\mathcal O}_{(\lambda+\rho)^*}\subset T^*K$, where ${\mathcal O}_{-(\lambda+\rho)}$ is a coadjoint orbit through $Y=-(\lambda+\rho)$, with $\lambda \in \hat K$ a highest weight and $\rho$ the Weyl vector given by the half-sum of the positive roots of $Lie(K)\otimes {\mathbb C}$. Consider the two natural projections $K\times {\mathcal O}_{\lambda}\to {\mathcal O}_{\lambda}\subset Lie(K)$, given by the restriction to $K\times {\mathcal O}_{\lambda}$ of moment maps $\mu,\tilde\mu$  for the right and left Hamiltonian of $K$ actions on $T^*K$, with $\mu(x,Y)=Y,\tilde \mu(x,Y)= \tilde Y=Ad_x(Y)$. Let $\omega$ be the standard symplectic form on $T^*K$. In \cite{Kirwin-Wu12}, it is shown that, for $Y\in Lie(K)$,
\[
\iota_Y^* \omega = -\iota_Y^*\mu^* \omega_{{\mathcal O}_Y} + \iota_Y^*\tilde \mu^* \omega_{{\mathcal O}_Y},
\]
where $\iota_Y: K\times {\mathcal O}_{Y}\to T^*K$ is the inclusion and $\omega_{{\mathcal O}_Y}$ is the standard Kirillov symplectic form on ${\mathcal O}_{Y}$.

The moment maps for the left and right actions of $K$ on $T^*K$ induce a fibration $(\mu,\tilde{\mu}):K\times \mathcal{O}_{-(\lambda+\rho)}\rightarrow \mathcal{O}^-_{\lambda+\rho} \times \mathcal{O}_{-(\lambda+\rho)}$ (the fibers are Cartan tori), where ${\mathcal O}_{(\lambda+\rho)}^-$ denotes ${\mathcal O}_{(\lambda+\rho)}$ equipped with the
negative of the usual Kirillov symplectic form \cite{Kirwin-Wu12}. The natural contribution to the quantization of $T^*K$ in the momentum polarization of a Bohr--Sommerfeld fiber $K\times {\mathcal O}_{-(\lambda+\rho)}$ is determined by half-form corrected Borel--Weil--Bott theory. Recall that in Borel--Weil--Bott theory the bundle of half-forms corresponds to the weight $-\rho$, so that in half-form quantization the representation $V_\lambda$ is associated to the co-adjoint orbit ${\mathcal O}_{(\lambda+\rho)}$. Moreover, note that in the present case over
${\mathcal O}_{(\lambda+\rho)}^-$ one has the negative of the usual Kirillov symplectic form \cite{Kirwin-Wu12} and that therefore one gets $V_{\lambda\otimes \lambda^*}$ associated to $K\times {\mathcal O}_{-(\lambda+\rho)}$. The contribution to the quantization of the Bohr-Sommerfeld fiber $K\times {\mathcal O}_{-(\lambda+\rho)}$ is therefore given by $V^{mom}_\lambda\simeq V_{\lambda\otimes \lambda^*}$.

One then has, in agreement with the above, a decomposition of the Hilbert space of quantum states for the momentum polarization, ${\mathcal{H}}^{mom}$, as
\[
{{\mathcal{H}}}^{mom} \simeq \oplus_{\lambda\in \hat K} V^{mom}_\lambda.
\]
Since the time evolution operator $e^{-i\tau \hat h}$ intertwines the $K\times K$ actions on each $\mathcal{H}_{\tau,h}$, we see that the quantum states in $V^{mom}_\lambda$ appear as the $\tau_2\to\infty$ time evolution of the quantum states in $V_{\lambda\otimes \lambda^*}\subset {{\mathcal{H}}}_{Sch}$.

The natural quantization in the momentum space polarization of an $Ad-$invariant function  $f(Y)$  is therefore given simply in terms of multiplication operators. That is, the quantum operator $\hat f$ acts on the quantum state localized at the Bohr--Sommerfeld fiber $K\times {\mathcal O}_{-(\lambda+\rho)}$ by multiplication by  $f(-(\lambda+\rho))$. The momentum space quantization $Q^{mom}(h)$ for the function $h$ acting on ${{\mathcal{H}}}^{mom}$, is then
\[
Q^{mom}(h)_{|_{V^{mom}_ \lambda}} = h(-(\lambda+\rho))\cdot Id_{V^{mom}_\lambda}.
\]
This motivates the definition of a quantization $Q(h)$ of the function $h$ acting on ${\mathcal H}_{Sch}$, by letting $e^{-i\tau \hat h}$ intertwine the actions of $Q^{mom}(h)$ and of $Q(h)$, as  $\tau_2\to\infty$. That is, we define $Q(h): {\mathcal H}_{Sch}\to {\mathcal H}_{Sch}$ by
\begin{equation}
Q(h)\cdot R^\lambda_{ij}(x)\otimes \sqrt{dx} = h(-(\lambda+\rho)) R^\lambda_{ij}(x)\otimes \sqrt{dx},\label{schquant}
\end{equation}
where $R^\lambda_{ij}$ is a matrix element for $\lambda\in \hat K.$

Recall that the Schr\"odinger--Duflo quantization \cite{Duflo} is an extension of the Schr\"odinger quantization to $Ad-$invariant functions of $Y$ with the property that it gives an associative algebra isomorphism between the space of $Ad-$invariant functions on $Lie(K)$ and the corresponding space of quantized operators on $L^2(K,dx)$. (It has recently been applied in the context of loop quantum gravity \cite{Sahlmann-Thiemann}. For a recent review see \cite{CR}.) While the momentum space quantization operators we defined above clearly share this property with the Duflo operators, an exact comparison between the spectrum of these two operators, for general $h$, does not seem to exist in the literature.

In the quadratic case, $h(Y)=\frac{1}{2}|Y|^{2}$, one has $Q^{SD}(h)=-\frac{1}{2}\Delta + \frac{|\rho|^2}{2},$ where $\Delta$ is the Laplace operator on $K$ for the bi-invariant metric  \cite{Mein}. Therefore, since the eigenvalues of $-\Delta$ are the quadratic Casimirs $C_2(\lambda)=(\lambda+\rho)^2 -\rho^2$, we have
\[
Q^{SD}(h)\cdot R^\lambda_{ij}(x)\otimes \sqrt{dx} = Q(h)\cdot R^\lambda_{ij}(x)\otimes \sqrt{dx},
\]
so that $Q^{SD}(h)$ and $Q(h)$ are the same operator on ${{\mathcal{H}}}_{Sch}$, in this case.

In this paper, we will consider the quantization defined by the operators $Q(h)$ in (\ref{schquant}), inducing (\ref{heisenberg}), since this is the choice leading to asymptotic unitarity in the limit $\tau_2\to +\infty$ in Section \ref{sasex}.

\smallskip
\noindent\textbf{3. Representation of ${\mathcal A}_0$ on $\mathcal{H}_{\tau,h}$.} (See Theorem \ref{316}.) The time-$(-\tau)$ Heisenberg evolution in (\ref{heisenberg}) can then be composed with the time-$(+\tau)$ Kostant--Souriau evolution (\ref{hhat})
\begin{equation}
f\mapsto e^{-i\tau\hat{h}}\circ e^{i\tau Q(h)}\circ f\circ e^{-i\tau Q(h)}\circ e^{i\tau\hat{h}},\label{bigevol}
\end{equation}
in the hope of obtaining a $^{\ast}$-representation $\gamma_{\tau,h}$ of $C^{0}(K)$ on ${{\mathcal{H}}}_{\tau,h}$ which forms a covariant pair together with $\rho_{{\tau,h}_{|_{K\times \{e\}}}}$. This turns out to work in the quadratic case $h(Y)=\frac{1}{2}|Y|^{2},$ thus yielding a Stone--Von Neumann-type interpretation of the CST of Hall. That is, the CST of Hall intertwines two irreducible covariant pairs of representations of $(C^0(K),K)$ defined naturally by geometric quantization, and therefore its unitarity is a consequence of the Stone--Von Neumann--Mackey theorem.

\section{Infinite-dimensional family of K\"ahler structures on $T^{*}K$ and Thiemann rays}
\label{s21}

In this section, we first review some basic facts concerning the geometric quantization of the cotangent bundle $T^{\ast}K$ of a compact Lie group $K$ with standard symplectic form. Then, we recall the infinite-dimensional family of K\"{a}hler polarizations considered in \cite{kirwin-mourao-nunes12} and define the Thiemann rays and their lift to the quantum bundle.

Let $\operatorname{dim}K=n$ and let $B$ be an $Ad$-invariant inner product on ${\mathfrak{k}}=\mathrm{Lie}(K)$ which induces the bi-invariant metric $\gamma$ on $K$ such that the associated Haar measure $dx$ has unit volume. We will henceforth identify $\mathfrak k\cong{\mathfrak k}^*$ via $B$. Let $\{X_{j}\}_{j=1,\dots,n}$ be an oriented orthonormal basis of left-invariant vector fields on $K$ and let $\{y^{j}\}_{j=1,\dots,n}$ be the corresponding coordinates on $\mathfrak{k}^{\ast}\cong\mathfrak{k}$. Let $\{w^{j}\}_{j=1,\dots,n}$ be the basis of left-invariant $1$-forms on $K$ dual to the vector fields $X_{j}$ such that $dx=w^{1}\wedge\cdots\wedge w^{n}$. We will denote their pullbacks to $T^{\ast}K$ along the canonical projection by $w^{j}$ as well. Similarly, we will need the coordinates $\{\tilde{y}^{i}\}_{i=1,\dots,n}$ on ${\mathfrak{k}}^{\ast}$ corresponding to a basis of right-invariant vector fields. Notice that $\{y^{j}\}_{j=1,\dots,n}$ and $\{\tilde{y}^{i}\}_{i=1,\dots,n}$ are the components of the moment maps $\mu,\tilde\mu$ for the right and left Hamiltonian actions of $K$ on $T^*K$, respectively. Consider the canonical symplectic structure on $T^{\ast}K$ given by
\[
\omega=-d\theta,
\]
where $\theta=\sum_{i=1}^{n}y^{i}w^{i}$ is the canonical $1$-form.

Let $K_{{\mathbb{C}}}$ be the complexification of $K$. Let ${\mathcal{C}}$ denote analytic continuation of functions from $K$ to $K_{{\mathbb{C}}}$. Recall the \textit{coherent state transform} (CST) of Hall
\begin{align*}
C_{t}:L^{2}(K,dx)  &  \rightarrow{{\mathcal{H}}}L^{2}(K_{{\mathbb{C}}},d\nu_{t})\\
f  &  \mapsto C_{t}(f)={\mathcal{C}}\circ e^{\frac{t}{2}\Delta}f,
\end{align*}
where $\Delta$ is the Laplacian for the metric $\gamma$, $t>0$ and ${{\mathcal{H}}}L^{2}(K_{{\mathbb{C}}},d\nu_{t})$ denotes the space of holomorphic functions on $K_{{\mathbb{C}}}$ which are square integrable with respect to the so-called averaged heat kernel measure $d\nu_{t}$ (see Theorem \ref{tautwopositive} below for a precise formula for $d\nu_{t}$). Hall proves:

\begin{theorem}
\cite{Hall94} For all $t>0$, $C_{t}$ is a unitary isomorphism of Hilbert spaces.
\end{theorem}

In \cite{Hall02,Florentino-Matias-Mourao-Nunes05,Florentino-Matias-Mourao-Nunes06,kirwin-mourao-nunes12}, the geometric quantization of $T^{*}K$ was related to the CST. In the present paper we continue the above study: namely we use the extension of the Thiemann complexifier method \cite{Thiemann96} to the geometric quantization of $T^{*}K$ to relate the unitarity of the (generalized) CSTs considered in \cite{kirwin-mourao-nunes12} to the Mackey-Stone-Von Neumann theorem.

The prequantum bundle $L$ is the trivial bundle $T^{\ast}K\times{{\mathbb{C}}}$, so its sections are just functions on $T^{\ast}K$. The half-form quantization of $T^{\ast}K$ in the vertical polarization\footnote{Throughout, $\left\langle V_{j},j=1,...,n\right\rangle _{\mathbb{C}}$ denotes the distribution whose fiber at a point $m$ is the complex span of the vectors
$V_{j}(m)\in T_{m}M\otimes{\mathbb{C}},j=1,...,n.$}
\[
{{\mathcal{P}}}_{Sch}=\langle\frac{\partial}{\partial y^{i}},i=1,\dots,n\rangle_{{\mathbb{C}}}
\]
then produces
\[
{{\mathcal{H}}}_{Sch}=\{f\otimes\sqrt{dx},f\in L^{2}(K,dx)\}\cong L^{2}(K,dx),
\]
where $dx=w^{1}\wedge\cdots\wedge w^{n}$ is the normalized Haar measure and where we write the pullback of $f\in L^{2}(K,dx)$ to $T^{\ast}K$ also by $f$.

As in \cite{kirwin-mourao-nunes12}, let us consider an infinite-dimensional family of K\"{a}hler structures $(T^{\ast}K,\omega,J_{\tau,h})$, labeled by a pair $(\tau,h)$, where $\tau\in{{{\mathbb{C}}}^{+}}$ and the function
\begin{equation}
h\ :\ K\times{\mathfrak{k}}\ \rightarrow\ {{\mathbb{R}}}, \label{aha}
\end{equation}
satisfies the following properties:
\begin{t_enumerate}
\item $h(x,Y)$ is an $Ad$-invariant smooth function depending only on $Y\in{\mathfrak{k}}$, that is,
it is a $K\times K$-invariant function on $T^*K$
\item the Hessian $H(Y)$ of $h$ is positive definite at every point $Y\in{\mathfrak{k}}$, and
\item the operator norm $||H(Y)||$ has nonzero lower bound.
\end{t_enumerate}
Denote the set of such functions $h$ by $\operatorname{Conv}(K).$

We consider the following diffeomorphisms
\begin{equation}
\begin{array}
[c]{ccccc}
T^{\ast}K & \overset{\alpha_{h}}{\rightarrow} & T^{\ast}K & \overset
{\psi_{\tau}}{\rightarrow} & K_{{\mathbb{C}}}\\
(x,Y) & \mapsto & (x,u(Y)) & \mapsto & xe^{\tau u(Y)},
\end{array}
\label{2diff}
\end{equation}
where $\alpha_{h}$ is the Legendre transform defined by $h$ (see \cite{Neeb00,kirwin-mourao-nunes12}) given by
\[
\alpha_{h}(x,Y)=\left(  x,\sum_{j=1}^{n}u^{j}(Y)T_{j}\right)  =\left(x,\sum_{j=1}^{n}\frac{\partial h}{\partial y_{j}}\ T_{j}\right)
\]
where $u$ denotes the gradient of $h$ and $\{T_j\}_{j=1,\dots,n}$ is an orthonormal basis for $\mathfrak k$. The diffeomorphism $\psi_{\tau}$ is studied in \cite{Neeb00,Florentino-Matias-Mourao-Nunes05,Florentino-Matias-Mourao-Nunes06,Hall-Kirwin,Lempert-Szoke10}. Let $J_{\tau,h}$ be the complex structure induced on $T^{\ast}K$ by the
diffeomorphism
\begin{align}
\Psi_{\tau,h}\  &  :=\psi_{\tau}\circ\alpha_{h}:\ T^{\ast}K\rightarrow K_{\mathbb{C}}\label{newdiff}\\
\Psi_{{\tau,h}}(x,Y)  &  =xe^{\tau u},\nonumber
\end{align}
i.e. the unique complex structure on $T^{\ast}K$ for which the map in (\ref{newdiff}) is a biholomorphism.

\begin{theorem}\cite{Neeb00,kirwin-mourao-nunes12}
\label{kahler}
For any $\tau\in{\mathbb{C}}^{+}$, the pair $(\omega,J_{\tau,h})$ defines a K\"ahler structure on $T^{*}K$, with K\"ahler potential
\begin{equation}
\label{kapo}\kappa(Y)=2\tau_{2}(B(Y,u(Y))-h(Y)).
\end{equation}
In particular, the corresponding K\"ahler polarization ${\mathcal{P}}_{\tau,h}:=T^{(1,0)}T^{*}K$ is positive.
\end{theorem}

Thus, the polarizations in the interior ${\mathcal T}$ of the family $\overline {\mathcal T}$ in (\ref{compact}) are K\"ahler. Note that we will say an object (function, form, etc...) is covariantly constant along a polarization $\mathcal{P}$ if its derivative along every vector field in $\overline{\mathcal{P}}$ is zero, so that in particular, covariantly constant along a K\"ahler polarization just means holomorphic with respect to that polarization.

A left-invariant holomorphic frame of type-$(1,0)$ forms is given by $\{\Omega_{\tau,h}^{i}\}_{i=1,\dots,n}$ where
\begin{equation}
\Omega_{\tau,h}^{j}=\sum_{k=1}^{n}\left[  e^{-\tau ad_{u(Y)}}\right]_{k}^{j}w^{k} +\left[\frac{1-e^{-\tau ad_{u(Y)}}}{ad_{u(Y)}}H(Y)\right]_{k}^{j}dy^{k} \label{tauforms}%
\end{equation}
(see \cite[Lemma 4.3]{kirwin-mourao-nunes12}). Let $\Omega_{\tau,h}=\Omega_{\tau,h}^{1}\wedge\dots\wedge\Omega_{\tau,h}^{n}$ denote the holomorphic left $K_{{\mathbb{C}}}$--invariant trivializing section of the canonical bundle ${\mathcal{K}}^{{\mathcal{P}}_{{\tau,{h}}}}=\bigwedge^{n}({\mathcal{P}}_{\tau,h})^{\ast}$ corresponding to ${\mathcal{P}}_{\tau,h}$.

\begin{theorem}\cite{kirwin-mourao-nunes12}
\label{tautwopositive}
Let $\tau\in{\mathbb{C}}^{+}\cup{{\mathbb{R}}}$. Then a section of $L\otimes\sqrt{{\mathcal{K}}^{{\mathcal{P}}_{\tau,h}}}$ is covariantly constant along the polarization $\overline {{\mathcal{P}}_{\tau,h}}$ if and only if it is of the form
\[
f(xe^{\tau u})\beta_{\tau,h}(Y)\otimes\sqrt{\Omega_{\tau,h}},
\]
for some function $f$, holomorphic in its argument, where
\[
\beta_{\tau,h}(Y)=e^{i\tau(B(u,Y)-h(Y))} \pi^{-\frac{n}{4}}\ .
\]
The Hilbert space $\mathcal{H}_{\tau,h}$ is then
\begin{equation}
\mathcal{H}_{\tau,h}=\left\{  s=f(xe^{\tau u})\beta_{\tau,h}(Y)\otimes \sqrt{\Omega_{\tau,h}}\ ,\,\,f\text{\ is\ holomorphic\ and\ }||s||\ <\infty \right\}  . \label{hilspth}
\end{equation}
\end{theorem}

Recall that the BKS (Blattner--Kostant--Sternberg) norm of the half-form $\sqrt{\Omega_{\tau,h}}$ is defined by comparing $\bar{\Omega}_{\tau,h}\wedge\Omega_{\tau,h}$ to $(2i)^n (-1)^{n(n-1)/2}$ times the Liouville form\footnote{In the last section of this paper, we will consider the semiclassical limit and hence need to include $\hbar.$ For this reason, we define the Liouville form to be
\[
\epsilon:=\frac{\omega^{n}}{\hbar^{n}n!}.
\]
As we do not yet need to keep track of factor of $\hbar,$ we set it equal to $1$ for the moment.} $\epsilon:=\omega^n/n!$. (The constant is chosen so that for the model case $M=\mathbb{C}$, the norm of $\sqrt{dz}$ is $1$.) Let $\eta(Y)$ be the $Ad$-invariant function defined for $Y$ in a chosen fixed Cartan subalgebra of $\mathfrak k$ by
\[
\eta(Y) = \Pi_{\alpha\in \Delta^+} \frac{\sinh \alpha(Y)}{\alpha(Y)},
\]
where $\Delta^+$ is the corresponding set of positive roots.

\begin{lemma}
\cite[Lemma 4.3]{kirwin-mourao-nunes12}\label{BKSnorm} The BKS norm of $\sqrt{\Omega_{\tau,h}}$ is
\[
\left\vert \sqrt{\Omega_{\tau,h}}\right\vert ^{2}=\tau_{2}^{n/2}\eta(\tau_{2}u(Y))\sqrt{\det H}.
\]
\end{lemma}

\begin{corollary}
The Hilbert space of square-integrable $\mathcal{P}_{\tau,h}$-polarized sections of $L\otimes\sqrt{\mathcal{K}^{\mathcal{P}_{\tau,h}}}$ is isomorphic to
\[
L_{J_{\tau,h}\text{-hol}}^{2}(T^{\ast}K,e^{-\kappa(Y)}\eta(\tau_{2}u(Y))\sqrt{\det H}dxdY)
\]
\end{corollary}

As in the case of quadratic $h$ studied in \cite{Hall-Kirwin}, the family $\{J_{\tau,h}\}$ of K\"{a}hler structures can be obtained from the vertical polarization by pushing forward by the complex-time flow of the Hamiltonian function $h$. For this reason, $h$ is sometimes called the \emph{Thiemann complexifier} \cite{Thiemann96}. We call a family of polarizations $\{{\mathcal P}_{\tau,h}, \tau \in {\mathbb C}^+\}$ a Thiemann ray of polarizations.

\begin{theorem}\cite{Hall-Kirwin, kirwin-mourao-nunes12}
\label{thmomegatau}
Let $\tau\in{\mathbb{C}}^{+}\cup{{\mathbb{R}}}$. Then
\begin{equation}
\label{thray}{\mathcal{P}}_{{\tau,h}}=e^{{\bar \tau}{\mathcal{L}}_{X_{h}}}{\mathcal{P}}_{Sch},
\end{equation}
as distributions.
\end{theorem}

\begin{remark}
Recall from \cite{Hall-Kirwin,kirwin-mourao-nunes12} that (\ref{thray}) can be interpreted literally as a power series in $\bar \tau$ if the operator $e^{{\bar \tau}{\mathcal{L}}_{X_{h}}}$ is applied to appropriate sections of $\mathcal{P}_{Sch}$ such as $\partial/\partial y^j$.
\end{remark}

The Hamiltonian flow of the function $h$ lifts to half-form corrected polarized sections as follows. Let
\[
\hat{h}=\left(  i\nabla_{X_{h}}+h\right)  \otimes1+1\otimes{{\mathcal{L}}}_{X_{h}}
\]
be the Kostant--Souriau prequantum operator for the Hamiltonian function $h$. Then,

\begin{theorem}\cite{kirwin-mourao-nunes12}
\label{ththrawf}
The densely defined operator $e^{-i\tau\hat{h}}:{\mathcal{H}}_{Sch}\rightarrow{\mathcal{H}}_{\tau,h},\ \tau\in{\mathbb{C}}^{+}$ is the $J_{\tau,h}-$analytic continuation of ${\mathcal{P}}_{Sch}$-polarized real-analytic sections,
\begin{equation}
e^{-i\tau\hat{h}}(f(x)\otimes\sqrt{\Omega_{0}})=f(xe^{\tau u})\beta_{\tau,h}(Y)\otimes\sqrt{\Omega_{\tau,h}}. \label{thraywf}
\end{equation}
\end{theorem}

The following result will be used in Sections \ref{s3} and \ref{sasex}.

\begin{theorem}\label{norms} [\cite{kirwin-mourao-nunes12}] Let $\lambda\in\hat K$ be an irreducible representation of $K$ of dimension $d_\lambda$ and let $\{R^{\lambda}_{ij}\}_{i,j=1,\dots,d_\lambda}$ be its matrix entries. Then the norms
\[
||R^\lambda_{ij}(xe^{\tau u})\beta_{\tau,h}(Y)\otimes \sqrt{\Omega_{\tau,h}}||_{{\mathcal H_{\tau,h}}},
\]
are independent of $i,j$ and of $\tau_1$. Moreover, these norms have a continuous limit as $\tau_2\to 0$, given by
\[
||R^\lambda_{ij}(x)\otimes \sqrt{dx}||_{{\mathcal H_{0,h}}}= ||R^\lambda_{ij}(x)||_{L^2(K,dx)} = d_\lambda^{-\frac12}.
\]
\end{theorem}

\section{The Mackey-Stone-Von Neumann theorem and quantization of $T^{*}K$}
\label{s3}

In the present section we will show how the Mackey-Stone-Von Neumann theorem helps in establishing a representation-theoretic proof of the unitarity of the CST.

\subsection{Covariant pairs and the Mackey theorem}
\label{s31}

Let $C^{0}(K)$ be the commutative $C^{\ast}$-algebra of continuous functions on $K$. A \emph{covariant pair} $(R,\gamma)$ for $(K,C^{0}(K))$ is a unitary representation $R$ of $K$ on an Hilbert space $\mathcal{H}$ and a $^{\ast}$-representation $\gamma$ of $C^0(K)$ on $\mathcal{H}$ such that
\[
R(x)\gamma(f)R^{\dagger}(x)=\gamma(x\cdot f)\ ,
\]
where $(x\cdot f)(x_{1})=f(x^{-1}x_{1})$, $x\in K, f\in C^0(K)$. The \emph{standard covariant pair} $(R^{st},\gamma^{st})$ for $(K,C^{0}(K))$ is given by the standard action on $\mathcal{H}^{st}:=L^{2}(K,dx)$, which is
\begin{align*}
(R^{st}(x_{1})\psi)(x)  &  =\psi(x_{1}^{-1}x)\\
\gamma^{st}(f)\psi(x)  &  =f(x)\psi(x),
\end{align*}
$x\in K, f\in C^0(K), \psi\in L^2(K,dx)$. (These notions are due to Mackey \cite{Mackey}.) The following theorem, the Mackey extension of the Stone--Von Neumann uniqueness result for the Heisenberg group, is the fundamental result lying at the root of our construction.

\begin{theorem}\cite{Mackey,Rosenberg04}
\label{thmac}
Any covariant pair $(R,\gamma)$ for $(K,C^{0}(K))$ is unitarily equivalent to a direct sum of at most countably many copies of the standard covariant pair.
\end{theorem}

\noindent Note that the Mackey theorem is valid for locally compact groups while we are considering only compact groups.

We define a {\it double covariant pair} for $(K\times K, C^0(K))$ to be a pair $(\widetilde R, \gamma)$, where $\widetilde{R}$ is a unitary representation of $K\times K$ on a Hilbert space ${\mathcal H}$ and $\gamma$ is a $*$-representation of $C^0(K)$ on ${\mathcal H}$, such that
\begin{equation}
\label{dcpair}\widetilde{R}(x_{1},x_{2})\gamma(f)\widetilde{R}^{\dagger}(x_{1},x_{2})=\gamma((x_{1},x_{2})\cdot f)\ ,
\end{equation}
where
\[
((x_{1},x_{2})\cdot f)(x)=f(x_{1}^{-1}xx_{2}),
\]
$x_1,x_2,x\in K, f\in C^0(K).$

Let the {\it standard double covariant pair} $(\widetilde{R}^{st},\gamma^{st})$ for $(K\times K, C^0(K))$ be given by
\begin{align*}
\mathcal{H}^{st}  &  =L^{2}(K,dx)\\
(\widetilde{R}^{st}(x_{1},x_{2})\psi)(x)  &  =\psi(x_{1}^{-1}xx_{2})\\
\gamma^{st}(f)\psi(x)  &  =f(x)\psi(x),
\end{align*}
where $x_1,x_2,x\in K, f\in C^0(K), \psi\in L^2(K,dx)$.

\begin{remark}
Note that a covariant pair for $(K\times K,C^0(K\times K))$ is different from a double covariant pair for $(K\times K,C^0(K))$.
\end{remark}

The following is a direct consequences of the Mackey theorem and of the standard representation.

\begin{corollary}
\label{cmac}
For a compact group $K$, any covariant pair $(R,\gamma)$ for $(K,C^0(K))$ has an extension to a double covariant pair $(\widetilde{R},\gamma)$ for $(K\times K, C^0(K))$ which is equivalent to the direct sum of at most countable many copies of the standard double covariant pair.
\end{corollary}

\begin{proof}
By Mackey's theorem the Hilbert space ${\mathcal H}$ for the covariant pair for $(K,C^0(K))$ decomposes, up to isomorphism, into at most countably many copies of $L^2(K,dx)$ with the standard action of $(K\times \{e\}, C^0(K))$. One can then define a double covariant pair by taking the direct sum of the standard action of $\{e\}\times K$ on each $L^2(K,dx)$ summand.
\end{proof}

\subsection{Covariant pairs versus Heisenberg evolution in complex time}
\label{s32b}

Let us return now to the Hilbert spaces $\mathcal{H}_{\tau,h}$ of $\mathcal{P}_{\tau,h}$-polarized sections (\ref{hilspth}). Geometric quantization defines the representation of observables that preserve the polarization. Recall that, for translation invariant polarizations on a symplectic vector space, geometric quantization defines an irreducible representation of the Heisenberg group on the Hilbert space of polarized sections, which leads to unitary equivalence of the corresponding quantizations via the Stone--Von Neumann uniqueness theorem.

The situation here is not quite the same because geometric quantization does not define a covariant pair for every $\mathcal{H}_{\tau,h}$, as the functions $f\in C^{\infty}(K)$ preserve only the Schr\"{o}dinger polarization. There are, however, several facts which bring this case very close to the former, with the additional bonus of allowing for the study of the unitary equivalence of quantizations for the infinite-dimensional family of polarizations $\overline{{\mathcal{T}}}$.

The first fact is that for every $(\tau,h),$ geometric quantization defines a representation $\widetilde{R}_{\tau,h}$ of $K\times K$ of the form (up to isomorphism)
\begin{equation}\label{deco}
\bigoplus_{\lambda\in \hat K} V_{\lambda\otimes \lambda^*} .
\end{equation}
Indeed, as we have shown in \cite{kirwin-mourao-nunes12}, the Hamiltonian functions generating the left and the right actions of $K$ on $K_{{\mathbb{C}}}$ preserve all polarizations $\mathcal{P}_{\tau,h}$ and the following result holds. Let $\{\hat{y}^{j},{\hat{\tilde{y}}}^{j}\}_{j=1,\dots,n}$ be the Kostant--Souriau prequantum operators corresponding to the coordinate functions $\{y^{j},\tilde{y}^{j}\}_{j=1,\dots,n}$. Recall the following results (Theorems 6.6 and 6.7 of \cite{kirwin-mourao-nunes12}).

\begin{theorem}\cite{kirwin-mourao-nunes12}
\label{thpreserves}
Let $\tau\in{\mathbb{C}}^{+}\cup{\mathbb{R}}$.
\begin{enumerate}
\item The action of $K\times K$ on ${\mathcal{H}}_{{\tau,h}}$ is generated by the operators $\{\hat{\tilde{y}}^{j},\hat{y}^{j}\}_{j=1,\dots,n}$, where the operators $\{\hat{y}^{j}\}_{j=1,\dots,n}$ generate the right $K$ action and the operators $\{\hat{\tilde{y}}^{j}\}_{j=1,\dots,n}$ generate the left $K$ action. Denoting this action by $\widetilde{R}_{\tau,h}$ we have
    \begin{equation}\label{repith}
    \widetilde{R}_{\tau,h}(x_{1},x_{2})s(x,Y)=s(x_{1}^{-1}xx_{2},Ad_{x_{2}^{-1}}Y)\ ,
    \end{equation}
and $\widetilde{R}_{\tau,h}$ satisfies (\ref{deco}).

\item The map $e^{-i\tau\hat{h}}$ in (\ref{thraywf}) intertwines the standard representation $\widetilde{R}^{st}$ with the representation $\widetilde{R}_{\tau,h}$. It is unitary if and only if $\tau\in{{\mathbb{R}}}$.
\end{enumerate}
\end{theorem}

Since the representations $V_{\lambda\otimes\lambda^{\ast}}$ of $K\times K$ are irreducible, from Schur's lemma we conclude that a unitary map $U$ intertwining $\widetilde{R}^{st}$ on $L^{2}(K,dx)$ with $\widetilde{R}_{\tau,h}$ is of the form
\begin{equation}
U_{\alpha}=\bigoplus_{\lambda\in\widehat{K}}\alpha(\lambda)\varphi_{\lambda
\otimes\lambda^{\ast}},\label{decointal}
\end{equation}
for any $\alpha\in Map(\widehat{K},U(1))$, where $\varphi_{\lambda \otimes\lambda^{\ast}}$ is a unitary intertwining operator between the subspace $V^\lambda_{0} =\{{tr}(A\lambda(x))\otimes \sqrt{dx},\,\,  A\in End(V_\lambda)\} \subset L^2(K,dx)\otimes \sqrt{dx}$ and the corresponding subspace
\begin{eqnarray}\nonumber
V_{\tau,h}^\lambda &=&\{ \operatorname{tr}(A\lambda(xe^{\tau u}))\beta_{\tau,h}(Y)\otimes \sqrt{\Omega_{\tau,h}},\,\,  A\in End(V_\lambda)\}=\\ \nonumber &=& e^{-i\tau \hat h} \, V_0^\lambda \subset {\mathcal H}_{\tau,h}.
\end{eqnarray}
The general extension of $\widetilde{R}_{\tau,h}$ to a covariant pair is then given by
\begin{equation}
\gamma_{\tau,h}(f)=U_{\alpha}\circ\gamma^{st}(f)\circ U_{\alpha}^{-1}\ ,\,\,\,f\in C^0(K).\label{gamaal}
\end{equation}

The second fact is that the ambiguity in defining the covariant pairs can be further restricted because all polarizations are in Thiemann rays starting at the vertical polarization (\ref{thray}) so that geometric quantization can tell us more regarding $\gamma_{\tau,h}(f),\ f\in C^{0}(K)$.

Let ${\mathcal{A}}$ denote the algebra generated by holomorphic functions on $K_{\mathbb C}$ of the form $f(g)=tr(\lambda(g)A)$, where $\lambda$ denotes an irreducible finite-dimensional representation of $K_{{\mathbb{C}}}$ on the vector space $V_{\lambda}$ and $A\in\mathrm{End}(V_{\lambda})$. We denote by $\mathcal{A}_{\tau,h}$ the pullback of ${\mathcal{A}}$ to $T^{\ast}K$ under $\Psi_{\tau,h}$. Since the symplectic form $\omega$ is of type $(1,1)$ with respect to$J_{\tau,h}$, these algebras are all abelian, Poisson, and multiplicative subalgebras of $C^{\infty}(T^{\ast}K,{{\mathbb{C}}}),$ and we have
\begin{equation}
\mathcal{A}_{\tau,h}=\Psi_{\tau,h}^* {\mathcal A}  =\left\{  e^{\tau X_{h}}f,\ f\in{\mathcal{A}}_0\right\}  =e^{\tau X_{h}}\mathcal{A}_{0}\ ,\label{ath0}
\end{equation}
where $\mathcal{A}_{0}$ is the algebra of functions on $K$ generated by matrix elements of irreducible representations, pulled back to $T^{\ast}K$. Recall that the expressions $e^{\tau X_{h}}f,\ f\in{\mathcal{A}}_0$ can be interpreted literally as convergent power series in $\tau$ \cite{kirwin-mourao-nunes12} and that $\mathcal{A}_{0}$ is dense in $C^{0}(K)$.

\begin{remark}
\label{rnotstar} Notice that $e^{\tau X_{h}}$ in (\ref{ath0}) establishes an algebraic and Poisson isomorphism between $\mathcal{A}_{0}$ and $\mathcal{A}_{\tau,h}$ but \emph{not} a $^{\ast}$-isomorphism. In fact $\mathcal{A}_{0}$ is a $^{\ast}$-subalgebra of $C^{\omega}(T^{\ast}K,{{\mathbb{C}}})$ and none of the algebras $\mathcal{A}_{\tau,h}$, for $\tau\in{{\mathbb{C}}}^{+}$, is a $^{\ast}$-subalgebra.
\end{remark}

\begin{proposition}
\label{ppxf}
\[
\overline{\mathcal{P}_{\tau,h}}=\langle X_{f},\ f\in\mathcal{A}_{\tau,h}\rangle_{{\mathbb{C}}}
\]
\end{proposition}

\begin{proof}
It is clear that $\mathcal{P}_{0}=\langle X_{f},\ f\in\mathcal{A}_{0}\rangle$. Then, from Theorem \ref{thmomegatau} we obtain
\[
\overline{\mathcal{P}_{\tau,h}}=e^{\tau {\mathcal L}_{X_{h}}}\left\langle X_{f},\ f\in\mathcal{A}_{0}\right\rangle =\left\langle X_{e^{\tau X_h}f},f\in\mathcal{A}_{0}\right\rangle =\left\langle X_{g},\ g\in\mathcal{A}_{\tau,h}\right\rangle.
\]
Here, we use $e^{\tau {\mathcal L}_{X_{h}}} X_{f}=X_{e^{\tau X_h}f}, f\in\mathcal{A}_{0}$. For real time the equality is a standard symplectic geometric fact and analyticity in $\tau$ guarantees that the equality also holds for complex time.
\end{proof}

\bigskip
The following proposition shows that $\mathcal{A}_{\tau,h}$ acts on $\mathcal{H}_{\tau,h}$.

\begin{proposition}
\label{pbetarep}
Let $f\in\mathcal{A}_{\tau,h}$. Then $f$ preserves the polarization
$\overline{\mathcal{P}_{\tau,h}}$ and the geometric prequantization of $f$, restricted to $\mathcal{H}_{\tau,h}$, defines a representation $\mu_{\tau,h}$ of $\mathcal{A}_{\tau,h}$ on $\mathcal{H}_{\tau,h}$ by
\begin{equation}
(\mu_{\tau,h}(f)s)(x,Y)=(\hat{f}s)(x,Y)=f(xe^{\tau u})\ s(x,Y).
\label{betarep}
\end{equation}
\end{proposition}

\begin{proof}
The first part is a direct consequence of the fact that $X_{f}\in{\Gamma}
(\overline{\mathcal{P}_{\tau,h}})$. For the second part notice that the Hamiltonian vector field $X_f$  is of type $(0,1)$ since $f$ is holomorphic and $\omega$ is of type $(1,1)$. Then, $\nabla_{X_{f}}\otimes1\ s=0$. Also, since the half-form $\sqrt{\Omega_{\tau,h}}$ is holomorphic \cite{kirwin-mourao-nunes12} we get $1\otimes{{\mathcal{L}}}_{X_{f}}\ s=0$, for all $s\in\mathcal{H}_{\tau,h}$.
\end{proof}

Note that the non-constant functions $f\in {\mathcal A}_{\tau,h}$ act on $\mathcal{H}_{\tau,h}$ as unbounded operators.

\begin{theorem}
\label{thlifts}
The operator $e^{-i\tau\hat{h}}\ :\ \mathcal{H}_{Sch}\rightarrow\mathcal{H}_{\tau,h}$ in (\ref{thraywf}) intertwines the representation $\gamma^{st}|_{\mathcal{A}_{0}}$ of $\mathcal{A}_{0}$ with the representation $\mu_{\tau,h}$ of $\mathcal{A}_{\tau,h}=e^{\tau X_{h}}\mathcal{A}_{0}$ on $\mathcal{H}_{\tau,h}$.
\end{theorem}

\begin{proof}
This is a simple consequence of (\ref{ath0}) and of Proposition \ref{pbetarep}.
\end{proof}

\bigskip
Given the representation $\mu_{\tau,h}$ of $\mathcal{A}_{\tau,h}$, one could try to extend it to a representation of $\mathcal{A}_{\tau,h}\oplus\overline{\mathcal{A}_{\tau,h}}$ so that it obeys the $^{\ast}$-relations
\[
\mu_{\tau,h}(\bar{f}):={\mu_{\tau,h}(f)}^{\dagger}\ ,\ f\in\mathcal{A}_{\tau,h},
\]
and then choose a factor ordering to quantize $(\mathcal{A}_{0}$, $\gamma_{\tau,h})$, in such a way that it extends to a $^{\ast}$-representation of $C^{0}(K)$ and such that the resulting pair $(\widetilde{R}_{\tau,h},\gamma_{\tau,h})$ is a double covariant pair. That is, given a function in $\mathcal{A}_{0}$ one would express it in terms of functions in $\mathcal{A}_{\tau,h}$ and of functions in $\overline{\mathcal{A}_{\tau,h}}$ and then quantize it as an operator on $\mathcal{H}_{\tau,h}$ using some factor ordering and obeying the $*$-relations.

We will, however, address the problem of quantizing $\mathcal{A}_{0}$ with respect to $\mathcal{P}_{\tau,h}$ in a way which will be explicitly compatible with the action of $K\times K$. We see from Theorem \ref{thlifts} that the operator $e^{-i\tau\hat{h}}$ intertwines the representation of $\mathcal{A} _{0}$ on $\mathcal{H}_{Sch}$ with the natural representation of the complex-time-evolved algebra of observables $\mathcal{A}_{\tau,h}=e^{\tau X_{h}}\mathcal{A}_{0}$ on the Hilbert space \footnote{This is always the case for observables which preserve a polarization and for their Hamiltonian evolution in real or complex time.} $\mathcal{H}_{\tau,h}$. As mentioned in Section \ref{intro} (see equations (\ref{heisenberg}) and (\ref{schquant})), in order to have $e^{-i\tau \hat h}$ defining instead a representation of ${\mathcal A}_0$ on ${\mathcal H}_{\tau,h}$, it is natural to  start from a representation of $\mathcal{A}_{-\tau,h}=e^{-\tau X_h} {\mathcal A}_0$ on ${\mathcal H}_{Sch}$. For this, we need to define a quantization $Q(h)$ of $h$ on ${\mathcal H}_{Sch}$ and then define the representation of ${\mathcal A}_{-\tau,h}$ on ${\mathcal H}_{Sch}$ via Heisenberg evolution, as in (\ref{heisenberg}). Even though $h$ does not preserve the vertical polarization, the $Ad$-invariance of $h$ and the properties of the momentum space polarization of \cite{Kirwin-Wu12}, ensure that there is a natural quantization $Q(h)$ of $h$ on $\mathcal{H}_{Sch}$, as we have defined in Section \ref{intro} in (\ref{schquant}).

\begin{theorem}
\label{39}
Let  $\tau_{2}>0$. The operators $e^{i\tau Q(h)}=e^{-\tau_{2}Q(h)}\cdot e^{i\tau_{1}Q(h)}$ on $\mathcal{H}_{Sch}$ are contraction operators and therefore bounded. They map square-integrable functions on $K$ to complex analytic ones.
\end{theorem}

\begin{proof}
Since $h$ is in $\operatorname{Conv}(K)$, it follows that $h(Y)\geq c_0 +  B(v_0,Y) + c_2 |Y|^2$ where $c_0\in{\mathbb R}$, $c_2>0$ and $v_0\in \mathfrak k$ are constants. Since $-\Delta \cdot R^\lambda_{ij}(x)=((\lambda+\rho)^2-\rho^2)R^\lambda_{ij}(x)$, for $R^\lambda_{ij}$ a matrix element of $\lambda\in \hat K$, we see that $||e^{-\tau_2 Q(h)}f||_{L^2(K,dx)}\leq|| e^{\alpha \Delta} f||_{L^2(K,dx)}$, for some constant $\alpha>0.$  The result then follows from well known properties of the heat kernel on $K$ \cite{Hall94}.
\end{proof}

We first define a representation $\tilde \gamma$ of the time-$(-\tau)$ Heisenberg evolution of the algebra $\mathcal{A}_{0}$ by
\begin{align}
\tilde{\gamma}\ :\mathcal{A}_{-\tau,h}=e^{-\tau X_{h}}\mathcal{A}_{0}  & \rightarrow\mathcal{B}(\mathcal{H}_{Sch})\nonumber\\
\tilde{\gamma}(e^{-\tau X_{h}}f)  &  =e^{i\tau Q(h)}\circ f\circ e^{-i\tau Q(h)}\ ,\,\,\,\,\,f\in\mathcal{A}_{0}\subset C^{0}(K). \label{gaztil}
\end{align}
We may now hope that by composing with $e^{-i\tau\hat{h}}$ we get a $^{\ast}$-representation of $e^{\tau X_{h}}\mathcal{A}_{-\tau,h}=\mathcal{A}_{0}$ on $\mathcal{H}_{\tau,h}$ that extends $\widetilde{R}_{\tau,h}$ to a double covariant pair. We will see that this is indeed the case for $h=\frac{|Y|^2}{2},$ but not for other (non-quadratic) choices of $h$.

\subsection{Unitarity of the KSH map and the Mackey Theorem}
\label{s32}

We have just seen that a natural candidate for a $^{\ast}$-representation of $C^{0}(K)$ that extends $\widetilde{R}_{\tau,h}$ to a covariant pair is induced from the transformation
\begin{align}
C_{\tau,h}\ :\ \mathcal{H}_{Sch}  &  \rightarrow\mathcal{H}_{\tau,h}\nonumber\\
C_{\tau,h}  &  =e^{-i\tau\hat{h}}\circ e^{i\tau Q(h)}, \label{vtauh}
\end{align}
where $Q(h)$ is the momentum space quantization of $h$ defined in (\ref{schquant}). The representation of $\mathcal{A}_{0}$ it induces on $\mathcal{H}_{\tau,h}$ is
\begin{equation}
\nu_{\tau,h}(f)=C_{\tau,h}\circ f\circ(C_{\tau,h})^{-1}. \label{repa0ta}
\end{equation}
We will call the map $C_{\tau,h}$ in (\ref{vtauh}) the \emph{KSH} (Kostant--Souriau--Heisenberg) map.

\begin{remark}
In the quadratic case $h(Y)=\frac12 |Y|^2$, $C_{\tau,h}$ coincides with the CST of Hall.
\end{remark}

\begin{proposition}\label{itspair}
The pair $(\widetilde{R}_{\tau,h},\nu_{\tau,h})$ of representations of $(K\times K,\mathcal{A}_{0})$ satisfies (\ref{dcpair}).
\end{proposition}

\begin{proof}
Since $Q(h)$ commutes with all the $X_{j}$ and from Theorem \ref{thpreserves} we conclude that $C_{\tau,h}$ intertwines $\widetilde{R}_{0}$ with $\widetilde{R}_{\tau,h}$ and $\gamma^{st}$ with $\nu_{\tau,h}$. This implies (\ref{dcpair}).
\end{proof}

\begin{theorem}\label{troika}
The following are equivalent.
\begin{enumerate}
\item $\nu_{\tau,h}$ is a $^{\ast}$-representation of $\mathcal{A}_{0}$.

\item $(\tilde R_{\tau,h},\nu_{\tau,h})$ extends to a double covariant pair for $(K\times K,C^0(K))$.

\item The KSH map $C_{\tau,h}$ is unitary.
\end{enumerate}
\end{theorem}

\begin{proof}
Given the $*$-representation of ${\mathcal A}_0$ on ${\mathcal H}_{Sch}$, we see immediately from (\ref{repa0ta}) that (1) and (3) are equivalent. To prove the equivalence to (2), note first that the $*$-representation of ${\mathcal A}_0$ extends to the standard  $*$-representation of $C^0(K)$. Since, from Proposition \ref{itspair} $(\tilde R_{\tau,h},\nu_{\tau,h})$ satisfies property (\ref{dcpair}), we see that $(\tilde R_{{\tau,h}_{|_{K\times \{e\}}}},\nu_{\tau,h})$ is a covariant pair if and only if (1) holds. The extension to a double covariant pair then follows as in Corollary \ref{cmac}.
\end{proof}

\bigskip
Recall (\ref{thraywf}) and let $h_{\lambda}=h(-(\lambda+\rho))$ be the eigenvalue of $Q(h)$ on $V_{\lambda\otimes\lambda^{\ast}}\subset {\mathcal H}_{Sch}$. Let $\chi_\lambda=\sum_{j=1}^{d_\lambda}R^\lambda_{jj}$ be the character of the representation associated to $\lambda$ and recall from Theorem \ref{norms}  that
$||\chi_\lambda(xe^{\tau u})\beta_{\tau,h}\otimes \sqrt{\Omega_{\tau,h}}||=\sqrt{d_\lambda}||R^\lambda_{ij}(xe^{\tau u})\beta_{\tau,h}\otimes \sqrt{\Omega_{\tau,h}}||$.
The following result follows immediately from \cite{kirwin-mourao-nunes12}.

\begin{proposition}\label{uuu}\cite{kirwin-mourao-nunes12}
The map $U^{\tau,h}\ :\ \mathcal{H}_{Sch}\ \rightarrow\mathcal{H}_{\tau,h}$, which we call the \emph{generalized $h$-CST}, given by
\begin{equation}
U^{\tau,h}(R^\lambda_{ij}(x)\otimes \sqrt{dx})=\frac{1}{{a_{\lambda}(\tau_{2})}}\,e^{i\tau_1 h_\lambda} R^\lambda_{ij}(xe^{\tau u})\beta_{\tau,h}(Y)\otimes\sqrt{\Omega_{\tau,h}}\ ,\label{uform}
\end{equation}
where
\[
a_{\lambda}(\tau_{2})=||\chi_{\lambda}(xe^{\tau u})\beta_{\tau,h}(Y) \otimes\sqrt{\Omega_{\tau,h}}||_{\mathcal{H}_{\tau,h}}, \lambda\in \hat K, i,j=1,\dots,d_\lambda,
\]
is a unitary map intertwining $\widetilde{R}_{0}$ with $\widetilde{R}_{\tau,h}$.
\end{proposition}

\begin{remark} The generalized $h$-CST in Proposition \ref{uuu} differs from the one in \cite{kirwin-mourao-nunes12} by the phase factor $e^{i\tau_1 h_\lambda}$. The present definition is more natural since for real $\tau$, that is on the boundary of $\overline {\mathcal T}$, it coincides with the expected intertwining operator for all $h\in \operatorname{Conv}(K)$.
\end{remark}

Using the notation of (\ref{deco}), we see that there is a decomposition under the action of $K\times K$ as
\[
U^{\tau,h}=\bigoplus_{\lambda\in\widehat{K}}\ U_{\lambda\otimes\lambda\ast}^{\tau,h}.
\]
Note that the generalized $h$-CST factorizes as
\[
{U}^{\tau,h}=e^{-i\tau\hat{h}}\circ E^{\tau,h},
\]
where
\begin{align*}
E^{\tau,h}:\ \mathcal{H}_{Sch} &  \rightarrow\mathcal{H}_{Sch}\\
E_{\lambda\otimes\lambda^{\ast}}^{\tau,h} &  =\ \frac{1}{{a_{\lambda}({\tau_{2}})}}\
e^{i\tau_1 h_\lambda}I_{V_{\lambda\otimes\lambda^{\ast}}}.
\end{align*}

\begin{remark} The generalized $h$-CST defines a double covariant pair on ${\mathcal H}_{\tau,h}$ via (\ref{repa0ta}).
\end{remark}

From (\ref{thraywf}) and (\ref{uform}) we obtain the following theorem.

\begin{theorem}\label{unit}
The KSH map (\ref{vtauh}) can be decomposed as
\[
C_{\tau,h}=\bigoplus_{\lambda\in\widehat{K}} a_{\lambda}(\tau_{2}) e^{-\tau_2 h_\lambda} U_{\lambda\otimes\lambda\ast}^{\tau,h}.
\]
Moreover, $C_{\tau,h}$ is unitary and defines a double covariant pair for $(K\times K,C^{0}(K))$ if and only if the spectrum of $Q(h)$ satisfies
\begin{equation}\label{growth}
e^{\tau_2 h_{\lambda}}= a_{\lambda}(\tau_{2}).
\end{equation}
\end{theorem}

The ability of the KSH map to produce covariant pairs is summarized in the following theorem.

\begin{theorem}
\label{316}
\mbox{ }
\begin{enumerate}
\item The KSH map (\ref{vtauh}) defines a covariant pair if $h(Y)=\frac{|Y|^{2}}{2}$.

\item If the function $h(Y)$ is not quadratic in $Y,$ then the representation of $\mathcal{A}_{0}\subset C^{0}(K)$ on ${\mathcal H}_{\tau,h}$ defined by $C_{\tau,h}$ is not a $^{\ast}$-representation and therefore does not extend to a double covariant pair.
\end{enumerate}
\end{theorem}

\begin{proof}
As we have described explicitly above, for the Hamiltonian $h(Y)=\frac{|Y|^{2}}{2}$ the transformation $C_{\tau,h}$ is the CST of Hall which is unitary. (Of course, Hall's CST was originally defined only for $\tau_1=0$ but it can be easily generalized to more general $\tau\in {\mathbb C}^+$ while keeping its unitarity properties.) Theorem \ref{troika} then gives (1). On the other hand, from the integral expressions for $a_\lambda(\tau_2)$ one knows that if $h$ is not quadratic then $a_\lambda(\tau_2)$ is not of the form (\ref{growth}) \cite{kirwin-mourao-nunes12}. (See also the explicit asymptotic evaluation of $a_\lambda(\tau_2)$ in Section \ref{sasex}.)
\end{proof}

In the next section we show, however, that the covariant pair and unitarity properties are recovered asymptotically both for small $\hbar$ and for large $\tau_{2}$.

\section{Asymptotic unitarity}
\label{sasex}

In this Section, we will study the asymptotic unitarity of the KSH map defined in (\ref{repa0ta}) by
studying the behavior of the norms of states in ${{\mathcal{H}}}_{\tau,h}$ in the limit $\tau_{2}\rightarrow +\infty$ and in the semiclassical limit $\hbar\rightarrow0$. In these limits, the norms of the states associated to matrix elements of $\lambda\in \hat K$, $R_{ij}^{\lambda}(xe^{\tau u})\beta_{\tau,h}(Y)\otimes\sqrt{\Omega_{\tau}}\in {\mathcal H}_{\tau,h}$, are seen to be very closely related to the action of the momentum space quantization operator $Q(h)$. In the quadratic case $h=\frac{1}{2}|Y|^{2}$ this happens for all $\tau$, and not just asymptotically, which translates in the appearance of the heat operator semigroup in Hall's CST. For more general Hamiltonians $h$, no such simple expression of the norms in ${{\mathcal{H}}}_{\tau,h}$ exists, so the generalized $h$-CSTs cannot be written in such an explicit closed form.

\bigskip
Recall that since $E^{\tau,h}$ intertwines the $K\times K$ action, its
action on an irreducible summand of $L^{2}(K,dx)=\bigoplus_{\rho\in\hat{K}}V_{\rho\otimes\rho^{\ast}}$ is multiplication by the (nonzero) constant $e^{i\tau_{1}h_{\lambda}}a_{\lambda}({\tau_{2}})^{-1}$, i.e. for any matrix entry $R_{ij}^{\lambda}$ we have
\begin{equation}
E^{\tau,h}R_{ij}^{\lambda}\otimes \sqrt{dx}=\frac{1}{a_{\lambda}({\tau_{2}})}\,e^{i\tau_1 h_\lambda}R_{ij}^{\lambda}(x)\otimes\sqrt{dx}\label{eq:Bansatz}
\end{equation}
where
\[h_\lambda:=h(-\hbar(\lambda+\rho))\]
is the eigenvalue of $Q(h)$ along the subspace $V_{\lambda\times\lambda^*}\subset\mathcal{H}_{Sch}.$

In the previous sections of this paper, we have worked in units such that $\hbar=1$. Since we are interested in the semiclassical behavior of $E^{\tau,h}$, we can no longer do this, and hence we explicitly keep track of factors of $\hbar.$ So, for example, the Liouville form is $\epsilon=(\omega/\hbar)^n/n!$, which means that the factor of
$\tau_2^{\frac{n}{2}}$ should be replaced by $\left(\tau_2\hbar\right)^{\frac{n}{2}}$ in Lemma \ref{BKSnorm}, and we have that the quantity
$a_{\lambda}(\tau_{2})$ appearing in (\ref{eq:Bansatz}) is $a_{\lambda}(\hbar,\tau)$. Also,
the prequantum connection gets multiplied by a factor of $\hbar^{-1}$ so that, in particular,
the Bohr--Sommerfeld conditions described in Section \ref{intro} become $Y=-\hbar (\lambda+\rho)$,
for $\lambda$ a highest weight. Then, from Theorem \ref{unit} and (\ref{growth}) we see that the KSH map
$C_{\tau,h}$ is unitary if and only if
\begin{equation}\label{last}
e^{\frac{\tau_2}{\hbar}h_\lambda} = a_\lambda(\hbar,\tau_2).
\end{equation}

We have from \cite{kirwin-mourao-nunes12}, and from the normalization of the Liouville form in Section \ref{s21},
\[
a_\lambda(\hbar,\tau_2)^2 = \left(\frac{\tau_2}{\pi\hbar}\right)^{\frac{n}{2}}  \int_{K\times \mathfrak k}\overline \chi_\lambda(xe^{\tau u}) \chi_\lambda(xe^{\tau u}) e^{-2\frac{\tau_2}{\hbar} (B(u,Y)-h(Y))}
\eta\left({\tau_2} u(Y)\right) \sqrt{\det H(Y)} dx dY.
\]
Using Weyl's orthogonality relations to perform the integral over $K$, we obtain
\begin{equation}\label{alambda}
a_\lambda(\hbar,\tau_2)^2 = {d_\lambda^{-1}} \left(\frac{\tau_2}{\pi\hbar}\right)^{\frac{n}{2}} \int_{\mathfrak k} \chi_\lambda(e^{2i\tau_2 u}) e^{-2\frac{\tau_2}{\hbar} (B(u,Y)-h(Y))} \eta\left({\tau_2} u(Y)\right) \sqrt{\det H(Y)} dY.
\end{equation}

\subsection{The $\tau_2\to +\infty$ asymptotics}

The next theorem shows that the unitarity condition for the KSH map is satisfied asymptotically in the large-$\tau_2$ limit. Recall that $\mathfrak k$ and ${\mathfrak k}^*$ are identified via the invariant form $B$ and that we denote the eigenvalue of $Q(h)$ along the subspace $V_{\lambda\otimes\lambda^*}\subset\mathcal{H}_{Sch}$ by $h_\lambda:=h(-\hbar(\lambda+\rho)).$

\begin{theorem}
\label{thm:asymp1}The KSH map $C_{\tau,h}$ is asymptotically unitary in the limit $\tau_2\to +\infty$. One has
\[
a_{\lambda}(\hbar,\tau_{2})^2\sim e^{2{\frac{\tau_2}{\hbar}}h_\lambda}(1+b_{1}(\lambda,\hbar)\tau_2^{-1}+O(\tau_2^{-2})),
\]
for some constant $b_{1}(\lambda,\hbar)$, as $\tau_2\rightarrow\infty$.
\end{theorem}

We first prove that $u$ maps the Cartan subalgebra $\mathfrak{h}$\ to itself.

\begin{lemma}
$X\in\mathfrak{h}$ \ implies $u(X)\in\mathfrak{h}$.
\end{lemma}

\begin{proof}
Recall from \cite{kirwin-mourao-nunes12} that $[u(X),X]=0$. For generic $X\in\mathfrak{h}$, this implies $u(X)\in\mathfrak{h}$. Since $u$ is continuous, $u(\mathfrak{h})\subset\mathfrak{h}$.
\end{proof}

\bigskip
\begin{proof}[Proof of Theorem \ref{thm:asymp1}]
This is an application of the Laplace approximation. (See \cite{Kirwin08}.) Using the facts that
\begin{enumerate}
\item for any $Ad-$invariant function $f$ and for any $z\in\mathbb{C}$ one
has
\[
\int_{\mathfrak{k}}f(Y)R_{ij}^{\lambda}(e^{zY})dY =\frac{\delta_{ij}}{d_{\lambda}}\int_{\mathfrak{k}}f(Y)\chi_{\lambda}(e^{zY})dY,
\]
\item for any $g,g^{\prime}\in K_{\mathbb{C}}$ one has
\[
\int_{K}R_{i^{\prime}j^{\prime}}^{\lambda^{\prime}}(g^{\prime}x^{-1})R_{ij}^{\lambda}(xg)dx=\frac{1}{d_{\lambda}}\delta_{\lambda\lambda^{\prime}}\delta_{ij^{\prime}}R_{i^{\prime}j}^{\lambda}(g^{\prime}g),
\]
\item and for any invariant function $f$
\[
\int_{\mathfrak{k}}fdY=\frac{1}{\left|W\right|\mathrm{vol}T} \int_{\mathfrak{h}}fP(X)^{2}dX,
\]
where $P(X):=\prod_{\Delta^{+}}\alpha(X)$ (see \cite[Chapter 9, Paragraph 6.3]{B}),
\end{enumerate}
we see that (\ref{alambda}) becomes
\begin{align*}
  a_\lambda(\hbar,\tau_2)^2= & \left(\frac{\tau_{2}}{\pi\hbar}\right)^{n/2}\frac{1}{\left|W\right|\mathrm{vol}T\, d_{\lambda}} \\
   & \quad \times \int_{\mathfrak{h}}\chi_{\lambda}(e^{2i\tau_{2}u(X)})e^{-2\tau_{2}(B(X,u(X))-h(X))/\hbar}\eta(\tau_{2}u(X))\sqrt{\mathrm{det}H(X)}P(X)^{2}dX.
\end{align*}
Using the Weyl character formula
\[
\chi_{\lambda}(e^{X})=\frac{\sum_{W}(-1)^{w}e^{iw(\lambda+\rho)(X)}}{2^{|\Delta^+|}
\prod_{\beta\in\Delta^{+}}\mathrm{sinh(}\frac{i}{2}\beta(X))}
\]
and the definition of $\eta$, (\ref{alambda}) becomes
\begin{align*}
a_\lambda(\hbar,\tau_2)^2=  & \frac{(\tau_{2})^{r/2}}{(\pi\hbar)^{n/2}}\frac{(2\tau_2)^{-|\Delta^+|}}{\left|W\right|\mathrm{vol}T\, d_{\lambda}} \\
& \quad \times \int_{\mathfrak{h}}\sum_{W}(-1)^{w}e^{-2\tau_{2}w(\lambda+\rho)(u)}e^{-2\tau_{2}(B(X,u(X))-h(X))/\hbar} \frac{\sqrt{\mathrm{det}H(X)}P(X)^{2}}{P(u)}dX\\
= & \frac{(2\tau_{2})^{r/2}}{(2\pi\hbar)^{n/2}}\frac{1}{\left|W\right|\mathrm{vol}T\, d_{\lambda}} \\
& \quad \times \int_{\mathfrak{h}}\sum_{W}(-1)^{w}e^{-2\tau_{2}(B(X,u(X))-h(X)+\hbar w(\lambda+\rho)(u))/\hbar}\frac{\sqrt{\mathrm{det}H(X)}P(X)^{2}}{P(u)}dX.
\end{align*}

It is easy to compute that the exponent has a critical point at
\[
X_{min}^{w}:=-\hbar w(\lambda+\rho),
\]
and that the exponent evaluated at $X_{min}$ yields $2\tau_{2}h(X_{min})/\hbar=2\tau_2 h_\lambda/\hbar$, since $h$ is $Ad$-invariant and therefore also $W$-invariant. Moreover, one easily computes that the Hessian of the exponent evaluated at $X_{min}$ is simply $H(X_{min})$.

Note also that as linear operators on the Lie algebra, $ad_Y = H(Y)^{-1} ad_{u(Y)}$ \cite{kirwin-mourao-nunes12}, so that
\[
\frac{\det H(X)}{\det_{\mathfrak{h}} H(X)} = \frac{P(u(X))^2}{P(X)^2}.
\]
Hence, applying Laplace's approximation, we obtain
\[
a_\lambda(\hbar,\tau_2)^2\sim\frac{(2\tau_{2})^{r/2}}{(2\pi\hbar)^{n/2}} \frac{(2\pi\hbar/2\tau_{2})^{r/2}}{\left|W\right|\mathrm{vol}T\, d_{\lambda}}\sum_{W}(-1)^{w}e^{2\tau_{2}h_\lambda/\hbar} P(X_{min}^{w}).
\]

 Hence,
\begin{eqnarray*}
\sum_{W}(-1)^{w}e^{2\tau_{2}h_\lambda/\hbar}P(X_{min}^{w}) & = & e^{2\tau_{2}h_\lambda/\hbar}\sum_{W}(-1)^{w}P(-\hbar w(\lambda+\rho))\\
 & = & \left|W\right|e^{2\tau_{2}h_\lambda/\hbar}P(-\hbar(\lambda+\rho)),
\end{eqnarray*}
where we used the identity $\sum_{W}(-1)^{w}P(wX)=\left|W\right|P(X).$

Note that, equivalently, one could have performed the sum over $W$ before taking the
Laplace approximation and describe
the result in terms of a single saddle point at $X_{min}=-\hbar (\lambda+\rho)$.

Since $P(-\hbar(\lambda+\rho))
=(-1)^{\left|\Delta^+ \right|}\hbar^{\left|\Delta^{+}\right|}d_{\lambda}P(\rho)$,
we obtain
\[
a_\lambda(\hbar,\tau_2)^2\sim\frac{P(\rho)}{(2\pi)^{\left|\Delta^{+}\right|} \mathrm{vol}T\, }\,e^{2\tau_{2}h_\lambda/\hbar}.
\]
Finally, since $1=\mathrm{vol}K=(2\pi)^{\left|\Delta^{+}\right|}\mathrm{vol}T/P(\rho)$ \cite{HS} we obtain
\[
a_\lambda(\hbar,\tau_2)^2\sim e^{2\tau_{2}h_\lambda/\hbar}
\]
as $\tau_2\rightarrow\infty$.
\end{proof}

\begin{remark}
Note that the values of $X_{min}^w$ are exactly the expected ones coming from the Bohr--Sommerfeld
conditions for $K\times {\mathcal O}_{-\hbar (\lambda+\rho)}$. We see that in the $\tau_2\to +\infty$ limit the main contribution to the norms of the quantum states in $V_{\lambda\otimes\lambda^*}\subset {\mathcal H}_{\tau,h}, \lambda\in \hat K$ does concentrate along the corresponding Bohr--Sommerfeld fiber $K\times {\mathcal O}_{-\hbar (\lambda+\rho)}\subset T^*K$. As we will see below, a similar behavior is observed for the semiclassical limit $\hbar\to 0$.
\end{remark}

\begin{proposition}
For $K=S^{1}$, we have, for $n\in {\mathbb Z}\cong \hat S^1$,
\[
a_n(\hbar,\tau_2) = e^{2\frac{\tau_2}{\hbar}h(-\hbar n)} (1+b_1(n,\hbar) \tau_2^{-1} + O(\tau_2^{-2})),
\]
where
\[
b_{1}(n,\hbar)= \frac{5(h^{(3)}(-\hbar n))^{2}-3h^{\prime\prime}(-\hbar n)h^{(4)}(- \hbar n)}{24(h^{\prime\prime}(-\hbar n))^3}.
\]
\end{proposition}

\begin{proof}
In the case $K=S^{1}$ we have $n=r=1$, $\chi_{n}(e^{X})=e^{inX}$, $\eta=P=1$, $d_{\lambda}=1$, $\mathfrak{k}=\mathfrak{h=}\mathbb{R}$, $u(y)=h^{\prime}(y)$, and $\det H=h^{\prime\prime}(y).$ Therefore
\[
a_{n}(\hbar,\tau_{2})=\sqrt{\frac{\tau_2}{\pi\hbar}}\int_{\mathbb{R}}e^{-2\frac{\tau_2}{\hbar} \left[(y+\hbar n)h^{\prime}(y)-h(y)\right]  }\sqrt{h^{\prime\prime}(y)}dy,\,\, n\in {\mathbb Z}.
\]
By \cite[Thm 1.1]{Kirwin08}, as $\tau_2\rightarrow +\infty$ the right-hand side is asymptotic to
\[
e^{2\frac{\tau_2}{\hbar}h(-\hbar n)}(1+b_{1}\tau_2^{-1}+O(\tau_2^{-2}))
\]
where a short computation yields the desired result.
\end{proof}

\subsection{Semiclassical asymptotics}

Let us now address the semiclassical limit $\hbar\to 0$. Instead of taking the Laplace approximation for the limit $\hbar\to 0$ directly in the expression (\ref{alambda}), it is useful to consider the integral
\begin{equation}
I_\lambda(\hbar,b,\tau_2) := \frac{1}{d_\lambda} \left(\frac{\tau_2}{\pi\hbar}\right)^{\frac{n}{2}} \int_{\mathfrak k}
\chi_\lambda(e^{2i\frac{b\tau_2}{\hbar} u}) e^{-2\frac{\tau_2}{\hbar} (B(u,Y)-h(Y))} \eta(\tau_2 Y) \sqrt{\det H(Y)} dY,
\end{equation}
so that
\[
a_\lambda(\hbar,\tau_2)^2=I_\lambda(\hbar,\hbar,\tau_2).
\]
We will {\it define} the semiclassical asymptotics by taking the Laplace approximation to $I_\lambda(\hbar,b,\tau_2)$ in the limit $\hbar\to 0$ and {\it then} by evaluating it at $b=\hbar$. This has the advantage of capturing already at leading order in $\hbar$ the contribution of the saddle point $X_{min}=-\hbar(\lambda+\rho)$, as in the proof of Theorem \ref{thm:asymp1}, which corresponds to the Bohr--Sommerfeld fiber $K\times{\mathcal O}_{-\hbar(\lambda+\rho)}$. Recall that such definition of the semiclassical asymptotics has been used in other contexts, see for example \cite{WITTEN}.

Let
\[
I_\lambda(\hbar,b,\tau_2) \sim F_\lambda (b,\tau_2) (1+O(\hbar^2)),
\]
in the limit $\hbar\to 0.$

\begin{lemma}
One has
\[
F(\hbar,\tau_2) = e^{2{\frac{\tau_2}{\hbar}}h_\lambda}.
\]
\end{lemma}

\begin{proof} This is the same calculation as in the proof of Theorem \ref{thm:asymp1}, except that we take the argument $e^{2i\frac{b\tau_2}{\hbar}u(X)}$ inside the character $\chi_\lambda$. The result then follows by a repetition of the same steps in that proof.
\end{proof}

Our final result is the following immediate corollary.

\begin{theorem}  The KSH map $C_{\tau,h}$ is asymptotically unitary in the semiclassical limit $\hbar\to 0$.
\end{theorem}

\bigskip
\noindent\textbf{\large Acknowledgments:} We thank Brian Hall for discussions.
This work was supported in part by the European Science Foundation (ESF) grant
``Interactions of Low-Dimensional Topology and Geometry with Mathematical Physics (ITGP)''.
The last two authors were supported by CAMGSD-LARSys through FCT Program POCTI-FEDER
and by the FCT project PTDC/MAT/119689/2010.

\end{document}